\documentclass[12pt]{amsart}

\usepackage{color}
\usepackage{enumitem}

\usepackage{amsmath,amssymb,setspace,nicefrac, yhmath, amscd,eucal}
\usepackage[active]{srcltx}
\usepackage[colorlinks, linkcolor=red, citecolor=blue, urlcolor=blue, hypertexnames=true]{hyperref}
\usepackage{amsrefs}
\usepackage{tikz}

\allowdisplaybreaks
\usepackage[all]{xy}

\newtheorem{thm}{Theorem}[section]
\newtheorem{cor}[thm]{Corollary}
\newtheorem{lem}[thm]{Lemma}

\newtheorem{definition}[thm]{Definition}

\newtheorem{proposition}[thm]{Proposition}

\theoremstyle{definition}

\title[]{An example of an infinite amenable group with the ISR property}
\author{Yongle Jiang}
\address{Y.J., School of Mathematical Sciences, Dalian University of Technology, Dalian, Liaoning, 116024, China}
\email{yonglejiang@dlut.edu.cn}

\author{Xiaoyan  Zhou*}
\address{X.Z., School of Date Science and Artificial Intelligence,
Dongbei University of Finance \& Economics, Dalian, Liaoning, 116025, China}
\email{xyzhou@dufe.edu.cn}
\thanks{$\ast$-corresponding author}

\date{\today}
\begin{document}

\begin{abstract}
Let $G$ be $S_{\mathbb{N}}$, the finitary permutation (i.e., permutations with finite support) group on the set of positive integers $\mathbb{N}$. We prove that $G$ has the invariant von Neumann subalgebras rigidity (ISR, for short) property as introduced in Amrutam-Jiang's work. More precisely, every $G$-invariant von Neumann subalgebra $P\subseteq L(G)$ is of the form $L(H)$ for some normal sugbroup $H\lhd G$ and in this case, $H=\{e\}, A_{\mathbb{N}}$ or $G$, where $A_{\mathbb{N}}$ denotes the finitary alternating group on $\mathbb{N}$, i.e., the subgroup of all even permutations in $S_{\mathbb{N}}$. This gives the first known example of an infinite amenable group with the ISR property.
\end{abstract}



\maketitle

\section{Introduction}

Let $G$ be a countable discrete group and $L(G)$ be the group von Neumann algebra generated by $G$. Note that $G$ acts on $L(G)$ naturally by conjugation. We say a von Neumann subalgebra $P\subseteq L(G)$ is \emph{invariant} if it is invariant as a set under the $G$-conjugation action. For example, if $H$ is a normal subgroup in $G$, then $L(H)$ is an invariant von Neumann subalgebra in $L(G)$. 

It is a natural question to classify all invariant von Neumann subalgebras in $L(G)$. 
In \cite[Corollary 4.8]{ab}, Alekseev and Brugger proved that for  a lattice subgroup $G$ in a higher rank simple real Lie group with trivial center, every invariant subfactor is  either $\mathbb{C}$ or has finite Jones's index in $L(G)$ using character rigidity techniques introduced in \cites{cp,pet}. For other classes of groups complementary to lattices in higher rank groups, including all non-abelian free groups, Chifan and Das made further progress (see \cite[Theorem 3.15, Theorem 3.16 and Corollary 3.17]{cd} for precise statements) on describing invariant subfactors but by deformation/rigidity techniques for array/quasi-cocycles on groups as introduced and studied in \cites{cs11, csu11,csu13,ckp15}.

Concerning the above question, the simplest answer one can hope is that every invariant von Neumann subalgebra is of the form $L(H)$ for some normal subgroup $H\lhd G$. In \cite{aj}, Amrutam and the first named author introduced ``the invariant von Neumann subalgebras rigidity" (ISR for short) property for a group $G$ once this best situation occurs, i.e., every invariant von Neumann subalgebra of $L(G)$ arises as the group von Neumann algebra of a normal subgroup of $G$. It is not hard to see that to search for groups with the ISR property, we may assume that $G$ is infinite (see Proposition \ref{prop: finite ISR groups}). The first known class of groups  with this property is due to Kalantar and Panagopoulos \cite{kp}. They proved that any irreducible lattice inside a connected semisimple Lie group with trivial center, no
non-trivial compact factors, and such that all its simple factors have real rank at least two has the ISR property. In fact, this result is the motivation for introducing the ISR property in \cite{aj}. In \cite{aj}, Amrutam and the first named author found other classes of groups with the ISR property, including all torsion-free hyperbolic groups with trivial amenable radical, torsion-free non-amenable groups with zero first $L^2$-Betti number under a mild assumption and finite direct product of groups from these classes. Subsequently, Chifan, Das and Sun \cite{cds} extended some results in \cite{aj} to much wider classes, e.g., all acylindricially hyperbolic groups with trivial amenable radical. Moreover, they conjectured that all non-amenable groups with trivial amenable radical should have the ISR property. Quite recently, in the remarkable work \cite{aho}, Amrutam, Hartman and Oppelmayer made important progress on this conjecture by proving that for any non-amenable group $G$, $L(G_a)$ is the unique maximal invariant amenable von Neumann subalgebra in $L(G)$, where $G_a$ denotes the amenable radical of $G$. Therefore, for a non-amenable group $G$ with trivial $G_a$, the only amenable invariant von Neumann subalgebra in $L(G)$ is the trivial subalgebra $\mathbb{C}$.

However, despite all these results, there are still two fundamental questions unanswered. One is whether the ISR property could be shared by any infinite amenable groups. Note that techniques in \cites{kp,aj,cds} are mainly developed to deal with non-amenable groups with trivial amenable radical.  The other question is whether the ISR property is preserved under W$^*$-equivalence. 

In this paper, we answer the above mentioned  two questions. For the first question, we show that the answer is positive by proving the following theorem.

\begin{thm}\label{thm: main thm}
Let $S_{\mathbb{N}}$ be the finitary permutation group on the set of positive integers $\mathbb{N}$.
Let $P\subseteq L(S_{\mathbb{N}})$ be an invariant von Neumann subalgebra in $L(S_{\mathbb{N}})$. Then $P=\mathbb{C}$, $L(A_{\mathbb{N}})$ or $L(S_{\mathbb{N}})$, where $A_{\mathbb{N}}$ denotes the finitary alternating group on $\mathbb{N}$.
\end{thm}
Combining it with \cite[Example 3.5]{aj}, we show that the second question has a negative answer.
\begin{cor}\label{cor: main corollary}
The ISR property is not preserved under W$^*$-equivalence, i.e., there exist two countable discrete groups $G$ and $H$ such that $L(G)\cong L(H)$, but exactly one of these groups has the ISR property.
\end{cor}

Let us briefly sketch the proof of Theorem \ref{thm: main thm}.

Let $G=S_{\mathbb{N}}$ and $P$ be an invariant von Neumann subalgebra in $L(G)$. Denote by $E: (L(G),\tau)\rightarrow P$ the $\tau$-preserving conditional expectation onto $P$, where $\tau$ denotes the canonical trace on $L(G)$. Note that $\tau(x)=\langle x\delta_e,\delta_e\rangle$ for all $x\in L(G)$ and we 
set $\langle a, b\rangle_{\tau}:=\tau(b^*a)$ for any $a,b\in L(G)$. As explained in \cite{aj}, to show that $G$ has the ISR property, we need to argue that $E(u_g)\in \mathbb{C}u_g$ for all $g\in G$. Hereafter we denote by $g$ the canonical unitary $u_g\in L(G)$ for simplicity. In other words, we need to argue that the trace (in the sense of Definition \ref{def: trace and characters}) $G\ni g\overset{\phi}{\mapsto}\tau(E(g)g^{-1})\in\mathbb{C}$ is the characteristic function supported on a normal subgroup of $G$.

First, we observe that for any $g\in G$, if we write $E(g)=\sum_{s\in G}c_ss$ for its Fourier expansion as an element in $L(G)$, then $\text{supp}(s)\subseteq \text{supp}(g)$ for all $c_s\neq 0$. Here, for any $g\in G$, we write $\text{supp}(g)=\{i\in\mathbb{N}:~gi\neq i\}$ for the support of $g$.
Then, it is not hard to see that for the transposition $(1~2)$, we have $E((1~2))=(1~2)$ or $E((1~2))=0$. In the first case, we show that $E(g)=g$ for all $g\in G$; equivalently, the trace $\phi$ defined above is the constant function one on $G$ and $P=L(G)$. So we may assume that $E((1~2))=0$.

The next step is to argue that $E((1~2~3))=0$ or $E((1~2~3))=(1~2~3)$. 
In the case $E((1~2~3))=(1~2~3)$, we argue that $P=L(A_{\mathbb{N}})$ since $A_{\mathbb{N}}$ is normally generated by all 3-cycles and has index two in $G$ (see Proposition \ref{prop: normal subgroups in S_N}). Finally, for the case $E((1~2~3))=0$, we observe that the classification of characters on $G$ by Thoma \cite{tho-classification}, or more precisely, Corollary \ref{cor: 3-cycle vanishing observation} actually implies $E(g)=0$ for all $e\neq g\in G$.
This new approach to study the ISR property via the classification of characters is the main innovation in this paper.

\subsection*{Structure of the paper} In Section \ref{section: preliminaries}, we record the classification of all normal subgroups of $S_{\mathbb{N}}$ and review basic facts on the notion of traces and characters on groups, then we prove a key observation on traces on $S_{\mathbb{N}}$ based on Thoma's classification of characters on $S_{\mathbb{N}}$, i.e., Corollary \ref{cor: 3-cycle vanishing observation}. In Section \ref{section: proofs}, we prove the main theorem and its corollary.

\subsection*{Notations and conventions}
\begin{itemize}
\item Denote by $\mathbb{N}$ the collection of all positive integers.
\item For any $A\subseteq \mathbb{N}$, we write $S_A=\{s\in S_{\mathbb{N}}: \text{supp}(s)\subseteq A\}$. Note that $S_A$ is a subgroup in $S_{\mathbb{N}}$.
\item Let $n\geq 1$. Set $[n]=\{1,2,\ldots, n\}$.
\item Let $n_1,\ldots, n_k$ be $k$-many distinct elements in $\mathbb{N}$. We write $s=(n_1~\cdots~n_k)\in S_{\mathbb{N}}$ to mean the $k$-cycle defined by $sn_i=n_{i+1~mod~k}$ and $si=i$ for all $i\not\in\{n_1,\ldots, n_k\}$.
    \item Let $k\geq 2$. For any two tuples $(n_1,n_2,\ldots, n_k), (m_1,m_2,\ldots, m_k)\in \mathbb{N}^k$, both of which have distinct entries, then we write $s=\left(\begin{smallmatrix}n_1&n_2&\cdots&n_k\\
    m_1&m_2&\cdots&m_k\end{smallmatrix} \right)$ to mean any element $s\in S_{\mathbb{N}}$ such that $sn_i=m_i$ for all  $1\leq i\leq k$. Since we only use the fact $s(n_1~n_2~\cdots~n_k)s^{-1}=(m_1~m_2~\cdots~m_k)$ in this paper, the ambiguity of $s$ would not cause problem.
\end{itemize}

\section{Preliminaries}\label{section: preliminaries}

The classification of all normal subgroups of $S_{\mathbb{N}}$ is well-known, see e.g., \cites{ono,su}. We are grateful to Prof. Mikael de la Salle for pointing out these references.
For the sake of completion,
we add a proof below. \begin{proposition}\label{prop: normal subgroups in S_N}
Let $N\lhd S_{\mathbb{N}}$ be a normal subgroup. Then either $N=\{e\}$, $N=A_{\mathbb{N}}$ or $N=S_{\mathbb{N}}$. Moreover, the subgroup index $[S_{\mathbb{N}}: A_{\mathbb{N}}]=2$.
\end{proposition}
\begin{proof}

To see the subgroup index is two, we just note that $S_{\mathbb{N}}=A_{\mathbb{N}}\sqcup (1~2)A_{\mathbb{N}}$. Indeed, recall that $[n]:=\{1,\ldots, n\}$, then we may take the union over $n\geq 5$ on both sides of the equality $S_{[n]}=A_{[n]}\sqcup (1~2)A_{[n]}$ for all $n\geq 5$.

Note that $S_{\mathbb{N}}=\cup_{n\geq 5}S_{[n]}$ and $A_{\mathbb{N}}=\cup_{n\geq 5}A_{[n]}$. Moreover, $A_{[n]}$ is a simple group and also  the unique proper non-trivial normal subgroup in $S_{[n]}$ for $n\geq 5$, see e.g., \cite[Theorem 3.11 and Exercise 3.21, Chapter 3]{rotman_gtm148}.

Since $N\cap S_{[n]}\lhd S_{[n]}$, we deduce that $N\cap S_{[n]}$ is either $\{e\}$, $A_{[n]}$ or $S_{[n]}$ for any $n\geq 5$.
If there exits some $n_0\geq 5$ such that $N\cap S_{[n_0]}=S_{[n_0]}$,
then $S_{[n_0]}\subseteq N$. Hence for all $m>n_0$, we have $S_{[n_0]}\subseteq N\cap S_{[m]}=S_{[m]}$ since $S_{[n_0]}\not\subseteq A_{[m]}$.
Therefore, $S_{[m]}\subseteq N$ and hence $N=N\cap S_{\mathbb{N}}=N\cap (\cup_{m\geq n_0}S_{[m]})=\cup_{m\geq n_0}(N\cap S_{[m]})=\cup_{m\geq n_0}S_{[m]}=S_{\mathbb{N}}$.

From now on, we may assume $N\cap S_{[n]}=A_{[n]}$ or $\{e\}$ for all $n\geq 5$.

Case 1: $N\cap S_{[n]}=A_{[n]}$ for all $n\geq 5$.

Then $A_{[n]}\subseteq N$ for all $n\geq 5$ and hence $A_{\mathbb{N}}=\cup_{n\geq 5}A_{[n]}\subseteq N\subseteq S_{\mathbb{N}}$. Clearly, $N\neq S_{\mathbb{N}}$, thus  the fact that $A_{\mathbb{N}}$ has index two in $S_{\mathbb{N}}$ implies $N=A_{\mathbb{N}}$.

Case 2: there exists some $n_0\geq 5$ such that $N\cap S_{[n_0]}=\{e\}$.

Take any $m>n_0$, then by our assumption, $N\cap S_{[m]}=\{e\}$ or $A_{[m]}$. Assume that $N\cap S_{[m]}=A_{[m]}$ for some $m>n_0$, then $\{e\}=N\cap S_{[n_0]}=(N\cap S_{[m]})\cap S_{[n_0]}=A_{[m]}\cap S_{[n_0]}=A_{[n_0]}$, a contradiction. Hence $N\cap S_{[m]}=\{e\}$ for all $m>n_0$. Thus, $N=\cup_{m\geq n_0}(N\cap S_{[m]})=\cup_{m\geq n_0}\{e\}=\{e\}$.
\end{proof}

Next, we recall some basic facts on the notion of traces on groups.
\begin{definition}\label{def: trace and characters}
A \textbf{trace} on a group $G$ is a function $\phi: G\rightarrow\mathbb{C}$ so that
\begin{item}
\item[(1)] $\phi$ is positive definite, i.e., 
\[\sum_{i=1}^n\sum_{j=1}^n\overline{\alpha_i}\alpha_j\phi(s_i^{-1}s_j)\geq 0\]
for all $n\in\mathbb{N}$ and any choice of elements $s_1,\ldots, s_n\in G$ and $\alpha_1,\ldots, \alpha_n\in\mathbb{C}$,
\item[(2)] $\phi$ is conjugation invariant, i.e., $\phi(sts^{-1})=\phi(t)$ for all elements $s, t\in G$ and 
\item[(3)] $\phi$ is normalized, i.e., $\phi(e)=1$.
\end{item}\\
A \textbf{character} on $G$ is an extreme point in the set of all traces on $G$.
\end{definition}
We remind the reader that some researchers prefer calling characters instead of  traces defined here (extremal or indecomposable characters instead of characters defined here respectively), see e.g., \cites{dm_ggd, dm_jfa,gk,pet}. Here, we have followed the terminology used in \cites{es,lv}.

Let $Tr(G)$ be the space of traces on $G$ equipped with the topology of pointwise convergence. Then $Tr(G)$ is a compact convex subset of $\ell^{\infty}(G)$. Let $Ch(G)$ be the set of characters on $G$. 
It is known that $Tr(G)$ is a metrizable Choquet simplex \cite{tho-simplex}. This means that the barycenter map 
\[Prob(Ch(G))\rightarrow Tr(G),~\mu\mapsto \phi_{\mu}:=\int_{Ch(G)}\chi d\mu(\chi).\]
is a continuous affine bijection, see e.g., \cite[Theorem 4.5, Chapter 4]{bo} and \cite{phe} for an overview of Choquet theory.

The classification of characters on various classes of groups have been an active topic, see \cite{bd_book} for an overview. 
In particular,
the following famous classical theorem due to Thoma is the main ingredient in our proof. See \cite[Chapter 4]{bo} for an exposition on this theorem and also \cite{gk} for an operator algebraic proof.

\begin{thm}[Thoma, \cite{tho-classification}]\label{thm: Thoma}
A character of the group $S_{\mathbb{N}}$ is of the form \[\chi(s)=\prod_{k=2}^{\infty}(\sum_{i=1}^{\infty}a_i^k+(-1)^{k-1}\sum_{j=1}^{\infty}b_j^k)^{m_k(s)}.\]
Here, $m_k(s)$ is the number of k-cycles in the permutation $s$ and the two sequences $(a_i)_{i=1}^{\infty}$, $(b_j)_{j=1}^{\infty}$ satisfy
\[a_1\geq a_2\geq \cdots\geq 0, ~b_1\geq b_2\geq \cdots\geq 0,~ \sum_{i=1}^{\infty}a_i+\sum_{j=1}^{\infty}b_j\leq 1.\]
\end{thm}
Recall that each permutation $s\in S_{\mathbb{N}}$ can be decomposed into disjoint cycles $s_1,s_2,\ldots, s_n\in S_{\mathbb{N}}$ for some $n\in\mathbb{N}$ such that $s=s_1s_2\cdots s_n$. A cycle $(n_1~n_2~\cdots~ n_k)$ has the length $k$ and is referred to as a $k$-cycle.

We observe the following fact holds.

\begin{cor}\label{cor: 3-cycle vanishing observation}
Let $\phi$ be a trace on $S_{\mathbb{N}}$. Assume that $\phi((1~2~3))=0$, then $\phi(s)=0$ for all $e\neq s\in S_{\mathbb{N}}$, i.e., $\phi$ is the regular character $\delta_e$.
\end{cor}
\begin{proof}
Note that $m_k((1~2~3))=1$ if $k=3$ and 0 for all $2\leq k\neq 3$.
Then by Theorem \ref{thm: Thoma}, we deduce that
\begin{align}\label{eq: formula for character value on 3-cycles}
\chi((1~2~3))=\sum_{i=1}^{\infty}a_i^3+\sum_{j=1}^{\infty}b_j^3,~\forall~\chi\in Ch(S_{\mathbb{N}}).   
\end{align}
Note that since $a_i\geq 0$ and $b_j\geq 0$ for all $i,j\geq 1$, the above formula implies that 
\begin{align}\label{eq: character value on 3-cycle}
\chi((1~2~3))\geq 0,~\forall~\chi\in Ch(S_{\mathbb{N}}).
\end{align}
First, we assume that $\phi$ is a character, i.e., $\phi\in Ch(S_{\mathbb{N}})$. Since  $\phi((1~2~3))=0$, we deduce that $a_i=b_j=0$ for all $i, j\geq 1$ from \eqref{eq: formula for character value on 3-cycles} and \eqref{eq: character value on 3-cycle}. Given any $e\neq s$, there exists some $k\geq 2$ such that $m_k(s)>0$, hence the $k$-th term in the product expression of $\phi(s)$ is 0, and therefore $\phi(s)=0$.

Now, let $\phi$ be a general trace.  Since $Tr(G)$ is a Choquet simplex, we may write $\phi=\int_{Ch(G)}\chi d\mu(\chi)$ for some $\mu\in Prob(Ch(G))$. Then, from $0=\phi((1~2~3))=\int_{Ch(G)}\chi((1~2~3))d\mu(\chi)$ and \eqref{eq: character value on 3-cycle}, we deduce that for $\mu$-a.e. $\chi\in Ch(G)$, we have $\chi((1~2~3))=0$. Thus $\chi(s)=0$ for all $e\neq s\in G$ and $\mu$-a.e. $\chi\in Ch(G)$ by what we have shown above. Thus, $\phi(s)=0$ for all $e\neq s\in G$.
\end{proof}

The following proposition explains why we focus on infinite groups when discussing the ISR property. In fact, it also follows directly from the proof of \cite[Proposition 3.1]{aj}, and we decide to include the proof here.

\begin{proposition}\label{prop: finite ISR groups}
Let $G$ be a finite group. Then $G$ has the ISR property if and only if $G$ is trivial or $G=C_2$, the cyclic group of order two.
\end{proposition}
\begin{proof}
To see the ``if" direction holds, observe that for $G=C_2$, $L(G)\cong \mathbb{C}\oplus \mathbb{C}$, which contains only two von Neumann subalgebras, i.e., the diagonal subalgebra $P=\{(c,c): \in\mathbb{C}\}\cong \mathbb{C}=L(\{e\})$ and $L(G)$ itself.

Next, we prove the ``only if" direction holds. Assume that the finite group $G$ has the ISR property, then the center $P:=\mathcal{Z}(L(G))=L(H)$ for some subgroup $H\lhd G$ since $P$ is clearly $G$-invariant. Notice that $p=\sum_{g\in G}g\in P$, we deduce that $G=H$ and thus, $L(G)=\mathcal{Z}(LG)$, which is abelian. Hence $G$ is abelian.

Since $G$ is abelian, every von Neumann subalgebra in $L(G)$ is invariant. Moreover, $L(G)\cong  \ell^{\infty}([n])$, where $n=\sharp G$. 
Assume that $n\geq 2$.
Notice that $P:=\{(x_i)\in \ell^{\infty}([n]): x_1=x_2\}$ is a von Neumann subalgebra and hence $P=L(H)$ for some $H\lhd G$. Notice that $\sharp H=n-1$, then from the Lagrange theorem, we get that $(n-1)\mid n$. Thus, $n=2$. Clearly, this finishes the proof.
\end{proof}

\section{Proofs}\label{section: proofs}

We will frequently use the simple fact that $st=ts$ for any $s,t\in S_{\mathbb{N}}$ with disjoint supports.

\begin{lem}\label{lem: lemma to control relative commutant}
Let $A\subseteq \mathbb{N}$ be an infinite subset. Then $L(S_A)'\cap L(S_{\mathbb{N}})=L(S_{\mathbb{N}\setminus A})$.
\end{lem}
\begin{proof}
``$\supseteq$" holds trivially. To see ``$\subseteq$" holds, it suffices to check that for any $s\in S_{\mathbb{N}}\setminus S_{\mathbb{N}\setminus A}$, we have $\sharp\{tst^{-1}: t\in S_A\}=\infty$.

Take any $s\in S_{\mathbb{N}}\setminus S_{\mathbb{N}\setminus A}$, then $\text{supp}(s)\cap A\neq\emptyset$. Fix any $i\in \text{supp}(s)\cap A$. 
Since $\sharp A=\infty$ and $\sharp\text{supp}(s)<\infty$, we may find an infinite strictly increasing sequence $\{j_k\}_{k\geq 1}$ in $A$ such that $j_1>\text{max}\{j: j\in \text{supp}(s)\}$. Then define the transpositions $t_k=(i~j_k)\in S_A$ for all $k\geq 1$. We are left to check that $t_kst_k^{-1}\neq t_{\ell}st_{\ell}^{-1}$ if $k\neq \ell$. Indeed, observe that 
$\text{supp}(t_kst_k^{-1})=t_k(\text{supp}(s))=t_k((\text{supp}(s)\setminus\{i\})\sqcup \{i\})=(\text{supp}(s)\setminus \{i\})\sqcup \{j_k\}$. It is clear that $\text{supp}(t_kst_k^{-1})\neq \text{supp}(t_{\ell}st_{\ell}^{-1})$ if $k\neq \ell$ by our choice of $\{j_k\}_{k\geq 1}$. Hence
$\sharp\{t_kst_k^{-1}: k\geq 1\}=\infty$.
\end{proof}

Let $G$ be any countable discrete group. Let $P\subseteq L(G)$ be a $G$-invariant von Neumann subalgebra. Denote by $E: (L(G),\tau)\rightarrow P$ the canonical $\tau$-preserving conditional expectation. Then set $\phi(g)=\tau(E(g)g^{-1})=\tau(E(g)E(g)^{-1})=||E(g)||^2_2$. It is well-known that $\phi$ is a trace on $G$, which has been used in \cite{ab}. For  convenience of the reader, we repeat the proof below.

\begin{proposition}\label{prop: phi is a trace}
Under the above notations, then $\phi: G\rightarrow \mathbb{C}$ is a trace on $G$.
\end{proposition}

\begin{proof}
It is clear that $\phi(e)=\tau(e)=1$.

Next, we show that $\phi$ is conjugation invariant.
Since $P$ is $G$-invariant, we deduce that $gE(x)g^{-1}=E(gxg^{-1})$ for all $x\in L(G)$ and $g\in G$. Hence, $\phi(sts^{-1})=\tau(E(sts^{-1})st^{-1}s^{-1})=\tau(sE(t)s^{-1}st^{-1}s^{-1})=\tau(E(t)t^{-1})=\phi(t)$.

Finally, we show that $\phi$ is positive definite. The proof has appeared in the proof of \cite[Theorem 12.2.15]{bo_book} or  \cite[Lemma 1]{choda}. Note that $E$ is a unital completely positive map by \cite[Theorem 1.5.10]{bo_book}. Thus, for any $s_1,\ldots, s_n\in G$ and $\alpha_1,\ldots, \alpha_n\in\mathbb{C}$, we have
\begin{align*}
\sum_{i=1}^n\sum_{j=1}^n\overline{\alpha_i}\alpha_j\phi(s_i^{-1}s_j)
&=\sum_{i=1}^n\sum_{j=1}^n\overline{\alpha_i}\alpha_j\tau(E(s_i^{-1}s_j)s_j^{-1}s_i)\\
&=\sum_{i=1}^n\sum_{j=1}^n\overline{\alpha_i}\alpha_j\tau(s_iE(s_i^{-1}s_j)s_j^{-1})\\
&=\sum_{i=1}^n\sum_{j=1}^n\overline{\alpha_i}\alpha_j\langle E(s_i^{-1}s_j)s_j^{-1}\delta_e, s_i^{-1}\delta_e\rangle\\
&=\sum_{i=1}^n\sum_{j=1}^n\langle E(s_i^{-1}s_j)\alpha_j \delta_{s_j^{-1}}, \alpha_i\delta_{s_i^{-1}}\rangle\\
&=\langle [E(s_i^{-1}s_j)]_{1\leq i,j\leq n}(\alpha_j\delta_{s_j^{-1}})_{1\leq j\leq n}, (\alpha_i\delta_{s_i^{-1}})_{1\leq i\leq n}\rangle\geq 0,
\end{align*}
where $(\alpha_j \delta_{s_j^{-1}})_{1\leq j\leq n}\in \ell^2(G)^{\oplus n}$ and $[E(s_i^{-1}s_j)]_{1\leq i,j\leq n}\in M_n(L(G))^{+}$ since $E$ is a unital completely positive map and $[(s_i^{-1}s_j)]_{1\leq i, j\leq n}\in M_n(L(G))^+$.
\end{proof}

From now on, we write $G$ for $S_{\mathbb{N}}$. 
The following proposition would be needed for the proof of Theorem \ref{thm: main thm}.

\begin{proposition}\label{prop: rule out symmetry}
Let $P\subseteq L(G)$ be a $G$-invariant von Neumann subalgebra. Let $E: (L(G),\tau)\rightarrow P$ be the canonical $\tau$-preserving conditional expectation. Suppose that $E((1~2~3))=\mu(1~2~3)+\theta(1~3~2)$ for some real numbers $\mu,\theta\in\mathbb{R}$, then $(\mu,\theta)\neq (\frac{1}{2},\pm \frac{1}{2})$.
\end{proposition}
\begin{proof}
First, we prove that $(\mu,\theta)\neq (\frac{1}{2}, \frac{1}{2})$.

Assume this does not hold, then $E((1~2~3))=\frac{1}{2}((1~2~3)+(1~3~2))\in P$ and $E((1~2~4))=(3~4)E(1~2~3)(3~4)=\frac{1}{2}((1~2~4)+(1~4~2))\in P$. In general, we have $E(i~j~k)=\frac{1}{2}((i~j~k)+(i~k~j))\in P$ for any three distinct natural numbers $i, j, k$.
Hence,
\begin{align}\label{eq: 1/2-case, product formula}
\begin{split}
P&\ni [(1~2~3)+(1~3~2)][(1~2~4)+(1~4~2)]
\\
&=(1~3)(2~4)+(1~4~3)+(2~4~3)+(1~4)(2~3).
\end{split}
\end{align}
Thus, taking the conjugation of the RHS in \eqref{eq: 1/2-case, product formula} by the transposition $(3~4)$ and noticing that $P$ is $G$-invariant, we get that
\begin{align}\label{eq: 1/2-case, ad(34)(product formula)}
    P\ni (1~4)(2~3)+(1~3~4)+(2~3~4)+(1~3)(2~4).
\end{align}
Adding the RHS of \eqref{eq: 1/2-case, product formula} to  \eqref{eq: 1/2-case, ad(34)(product formula)} and noticing that $(1~4~3)+(1~3~4), (2~3~4)+(2~4~3)\in P$, we get that
\begin{align}\label{eq: 1/2-case, sum of product of 2-cycles in P}
    (1~3)(2~4)+(1~4)(2~3)\in P.
\end{align}
Taking the conjugacy of the LHS of \eqref{eq: 1/2-case, sum of product of 2-cycles in P} by the transposition $(2~3)$, we get that
\begin{align}\label{eq: 1/2-case, ad(23) on sum of product of 2-cycles}
    (1~2)(3~4)+(1~4)(2~3)\in P.
\end{align}
Thus, by subtracting \eqref{eq: 1/2-case, ad(23) on sum of product of 2-cycles} from \eqref{eq: 1/2-case, sum of product of 2-cycles in P}, we get that
\begin{align}\label{eq: 1/2-case, subtraction of product of 2-cycles in P}
    (1~3)(2~4)-(1~2)(3~4)\in P.
\end{align}
Taking the conjugation of the LHS of \eqref{eq: 1/2-case, sum of product of 2-cycles in P} by the transposition $(2~4)$, we get that
\begin{align}\label{eq: 1/2-case, ad(24) of sum product of 2-cycles}
(1~3)(2~4)+(1~2)(3~4)\in P.
\end{align}
Then by adding \eqref{eq: 1/2-case, subtraction of product of 2-cycles in P} to \eqref{eq: 1/2-case, ad(24) of sum product of 2-cycles}, we deduce that $(1~3)(2~4)\in P$. Thus by taking the conjugation of $(1~3)(2~4)$ by $(3~4)$, we also have $(1~4)(2~3)\in P$. In view of \eqref{eq: 1/2-case, product formula}, this implies that $a:=(1~4~3)+(2~4~3)\in P$.
Then $(2~5)a(2~5)=(1~4~3)+(5~4~3)\in P$. So $(2~4~3)+(5~4~3)\in P$ after taking conjugation by $(1~2)$.
Hence, $a-(2~5)a(2~5)=(2~4~3)-(5~4~3)\in P$.
Therefore, $(2~4~3)\in P$. Thus, $(1~2~3)\in P$ and $E((1~2~3))=(1~2~3)$, a contradiction.

Next, let us show that $(\mu,\theta)\neq (\frac{1}{2},\frac{-1}{2})$.

Assume that $(\mu,\theta)=(\frac{1}{2},\frac{-1}{2})$ instead, then $E((1~2~3))=\frac{1}{2}((1~2~3)-(1~3~2))\in P$. Hence, $E((1~2~4))=\frac{1}{2}((1~2~4)-(1~4~2))\in P$ and more generally, $E((i~j~k))=\frac{1}{2}((i~j~k)-(i~k~j))\in P$ for any three distinct natural numbers $i, j, k$. Then,
\begin{align}\label{eq: -1/2-case, product formula}
\begin{split}
P&\ni [(1~2~3)-(1~3~2)][(1~2~4)-(1~4~2)]\\
&=(1~3)(2~4)-(1~4~3)-(2~4~3)+(1~4)(2~3).
\end{split}
\end{align}
Taking the conjugation of the RHS of \eqref{eq: -1/2-case, product formula} by the transposition $(2~3)$, we deduce that
\begin{align}\label{eq: -1/2-case, ad(23) of the product formula}
    P\ni (1~2)(3~4)-(1~4~2)-(3~4~2)+(1~4)(2~3).
\end{align}
Notice that $-(2~4~3)+(3~4~2)=-(2~4~3)+(2~3~4)\in P$, we may subtract \eqref{eq: -1/2-case, ad(23) of the product formula} from \eqref{eq: -1/2-case, product formula} to get that
\begin{align}\label{eq: -1/2-case, mixed product of 2-cycles in P}
(1~3)(2~4)-(1~2)(3~4)-(1~4~3)+(1~4~2)\in P.
\end{align}
Taking the conjugation of the LHS of \eqref{eq: -1/2-case, mixed product of 2-cycles in P} by (1~3), we get that
\begin{align}\label{eq: -1/2-case, 2nd mixed product of 2-cycles}
    (1~3)(2~4)-(2~3)(1~4)-(1~3~4)+(3~4~2)\in P.
\end{align}
Subtracting \eqref{eq: -1/2-case, 2nd mixed product of 2-cycles} from \eqref{eq: -1/2-case, mixed product of 2-cycles in P} and then noticing that $-(1~4~3)+(1~3~4)\in P$, we deduce that
\begin{align}\label{eq: -1/2-case, mixed minus 2nd mixed}
-(1~2)(3~4)+(2~3)(1~4)+(1~4~2)-(3~4~2)\in P.
\end{align}
Adding \eqref{eq: -1/2-case, ad(23) of the product formula} to \eqref{eq: -1/2-case, mixed minus 2nd mixed}, we get that
\begin{align}
    (2~3)(1~4)-(3~4~2)\in P.
\end{align}
In view of \eqref{eq: -1/2-case, ad(23) of the product formula}, this implies that $(1~2)(3~4)-(1~4~2)\in P$. 

Then
\begin{align*}
    0&=\langle (1~4~2)-E((1~4~2)), (1~2)(3~4)-(1~4~2))\rangle_{\tau}\\
    &=\langle (1~4~2)-\frac{1}{2}((1~4~2)-(1~2~4)), (1~2)(3~4)-(1~4~2)\rangle_{\tau}=\frac{-1}{2}.
\end{align*}
This yields a contradiction.
\end{proof}

We are ready for the proof of Theorem \ref{thm: main thm}.

\begin{proof}[Proof of Theorem \ref{thm: main thm}]
We assume the above notations.
To start the proof, we observe the following simple fact.
\begin{align}\label{eq: control support of E(s)}
    \forall~ e\neq s\in G, ~\text{we have}~E(s)\in L(S_{\text{supp}(s)}).
\end{align}
Indeed, $\forall~ t\in S_{\mathbb{N}\setminus \text{supp}(s)}$, we have $tst^{-1}=s$. Hence, $E(s)=E(tst^{-1})=tE(s)t^{-1}$ implies that $E(s)\in L(S_{\mathbb{N}\setminus \text{supp}(s)})'\cap G= L(S_{\text{supp}(s)})$, where the last equality follows by applying Lemma \ref{lem: lemma to control relative commutant} to the infinite set $A=\mathbb{N}\setminus \text{supp}(s)$.

From \eqref{eq: control support of E(s)}, we may write $E((1~2))=c+\mu(1~2)$ for some $c,\mu\in\mathbb{C}$. 
Hence, 
\begin{align*}
c+\mu(1~2)=E((1~2))&=E(E((1~2)))\\
&=E(c+\mu(1~2))=c+\mu E((1~2))=c+\mu(c+\mu (1~2)).    
\end{align*}
Equivalently, $c=c+c\mu$ and $\mu=\mu^2$. Hence, we get that either $\mu=0$ or $(\mu,c)=(1,0)$.

Case 1: $(\mu,c)=(1,0)$.

In this case, we have $E((1~2))=(1~2)$. Thus, for any $i\neq j\in\mathbb{N}$, set $s=\left(\begin{smallmatrix}1&2\\i&j\end{smallmatrix}\right)$, then  $E((i~j))=E(s (1~2)s^{-1})=s E((1~2))s^{-1}=(i~j)$. Thus, all transpositions belong to $P$. Since $S_{\mathbb{N}}$ is generated by transpositions, we deduce that $P=L(S_{\mathbb{N}})$.

Case 2: $\mu=0$.

We may write $E((1~2))=c$. Hence $c=\tau(E((1~2)))=\tau((1~2))=0$, i.e., \[E((1~2))=0.\]

Next, we show that $E((1~2~3))=(1~2~3)$ or 0. We split the proof into several steps.

\textbf{Step 1}: we show that $E((1~2~3))=\mu(1~2~3)+\theta(1~3~2)$ for some $\mu,\theta\in\mathbb{C}$.

To see this, we observe the following facts.
\begin{itemize}
    \item the three transpositions $(1~2)$, $(1~3)$ and $(2~3)$ belong to the same orbit for the conjugation action of the subgroup $\langle (1~2~3)\rangle \curvearrowright S_{[3]}$. Indeed, $(1~2~3)(1~3)(1~3~2)=(1~2)$, $(1~2~3)(1~2)(1~3~2)=(2~3)$.
    \item $E((1~2~3))$ is invariant under the conjugation of $(1~2~3)$ and $(1~3~2)$.
    \item $\tau(E((1~2~3)))=0$.
\end{itemize}
Combining the above facts with \eqref{eq: control support of E(s)}, we may write \[E((1~2~3))=c((1~2)+(1~3)+(2~3))+\mu(1~2~3)+\theta(1~3~2)\] for some $c,\mu,\theta\in\mathbb{C}$.
Then, we observe that $E((1~2))=0$ implies $\langle (1~2), x\rangle_{\tau}=0$ for all $x\in P$. In particular, plugging in $x=E((1~2~3))$, we deduce that $c=0$.
Therefore, 
\[E((1~2~3))=\mu(1~2~3)+\theta(1~3~2).\]
\textbf{Step 2}: we show that the pair $(\mu,\theta)\in \{(0, 0), (1, 0), (\frac{1}{2},\frac{1}{2}), (\frac{1}{2},\frac{-1}{2})\}$.

First, we observe that $\mu,\theta\in\mathbb{R}$. Indeed,
on the one hand,
\[E((1~3~2))=E((2~3)(1~2~3)(2~3))=(2~3)E((1~2~3))(2~3)=\mu(1~3~2)+\theta(1~2~3);\] on the other hand, $E((1~3~2))=E((1~2~3)^{-1})=E((1~2~3))^*=\overline{\mu}(1~3~2)+\overline{\theta}(1~2~3)$. Hence, we get that $\mu,\theta\in\mathbb{R}$.

Then, we have
\begin{align*}
    E((1~2~3))=E(E((1~2~3)))&=E(\mu(1~2~3)+\theta(1~3~2))=\mu E((1~2~3))+\theta E((1~3~2))\\
    &=\mu(\mu(1~2~3)+\theta(1~3~2))+\theta(\mu(1~3~2)+\theta(1~2~3))\\
    &=(\mu^2+\theta^2)(1~2~3)+2\mu\theta(1~3~2).
\end{align*}
By comparing it with the expression for $E((1~2~3))$, we deduce that $\mu=\mu^2+\theta^2$ and $\theta=2\mu\theta$ hold true. Hence, we get that
either 
\begin{itemize}
    \item[(i)] $\theta=0$, $\mu\in\{0, 1\}$ \text{or},
    \item[(ii)] $\mu=\frac{1}{2}$, $\theta=\pm \frac{1}{2}$.
\end{itemize}
In other words, the pair $(\mu,\theta)\in \{(0,0), (1, 0), (\frac{1}{2},\frac{1}{2}), (\frac{1}{2},\frac{-1}{2})\}$. 

Therefore, we deduce that $(\mu,\theta)\in\{(0, 0), (1,0)\}$ from Proposition \ref{prop: rule out symmetry}.

\textbf{Step 3}: we show that if $(\mu,\theta)=(0, 0)$, then $P=\mathbb{C}$; if $(\mu,\theta)=(1, 0)$, then $P=L(A_{\mathbb{N}})$.

Case 1: $(\mu,\theta)=(1, 0)$.

Then $E((1~2~3))=(1~2~3)\in P$. For any three distinct natural numbers $i, j, k$, set $s=\left(\begin{smallmatrix}1&2&3\\i&j&k\end{smallmatrix}\right)$, then we have $E((i~j~k))=E(s(1~2~3)s^{-1})=sE((1~2~3))s^{-1}=s(1~2~3)s^{-1}=(i~j~k)\in P$. Recall that $[n]:=\{1,\ldots, n\}$.
Since $A_{\mathbb{N}}=\cup_{n\geq 5}A_{[n]}$, it equals the normal subgroup of $S_{\mathbb{N}}$  generated by all 3-cycles \cite[Lemma 3.9, Chapter 3]{rotman_gtm148}, we deduce that $L(A_{\mathbb{N}})\subseteq P\subseteq L(S_{\mathbb{N}})$. Then notice that $E((1~2))=0$ implies  $E((1~2)s)=E((1~2))s=0$ for all $s\in A_{\mathbb{N}}$, hence $P=L(A_{\mathbb{N}})$ since $S_{\mathbb{N}}=A_{\mathbb{N}}\sqcup (1~2)A_{\mathbb{N}}$.

Case 2: $(\mu,\theta)=(0, 0)$.

In this case, we have $E((1~2~3))=0$. Hence $\phi((1~2~3))=0$ for the trace $\phi$ defined by $\phi(s)=\tau(E(s)s^{-1})$. Then $\phi(s)=0$ for all $e\neq s\in S_{\mathbb{N}}$ by Corollary \ref{cor: 3-cycle vanishing observation}. Hence $||E(s)||_2^2=\tau(E(s)E(s)^*)=\tau(E(s)E(s^{-1}))=\tau(E(s)s^{-1})=\phi(s)=0$. Thus, $E(s)=0$ for all $e\neq s\in S_{\mathbb{N}}$. Therefore, $P=\mathbb{C}$.
\end{proof}

Finally, we present the proof of Corollary \ref{cor: main corollary}.

\begin{proof}[Proof of Corollary \ref{cor: main corollary}]
Recall that we have proved in \cite[Example 3.5]{aj} that the group $H:=\mathbb{Z}\wr \mathbb{Z}$ does not have the ISR property. Observe that both $H$ and $G=S_{\mathbb{N}}$ are i.c.c. amenable groups, hence $L(G)\cong L(H)\cong \mathcal{R}$, the hyperfinite II$_1$ factor by Connes's theorem \cite{connes}. The proof is done since $S_{\mathbb{N}}$ has the ISR property by Theorem \ref{thm: main thm}.
\end{proof}

\subsection*{Acknowledgements}
Y. J. is partially supported by National Natural Science Foundation of China (Grant No. 12001081) and the Fundamental Research Funds for the Central Universities (Grant No.DUT19RC(3)075).
X. Z. is partially supported by National Natural Science Foundation of China (Grant No. 12001085) and the Scientific Research Fund of Liaoning Provincial Education Department (Grant No. LJKZ1047). We are grateful to Prof. Adam Skalski and Dr. Amrutam Tattwamasi for helpful comments. We also thank the anonymous referee for excellent comments which help improving the readability of the paper greatly.

\end{document}